\documentclass[11pt]{amsart}
\usepackage{graphicx}
\usepackage[space]{grffile}
\usepackage[T1]{fontenc}
\usepackage{color}
\usepackage[colorlinks,linkcolor=blue,citecolor=blue, pdfstartview=FitH]{hyperref}
\usepackage{backref}
\usepackage{amssymb}
\usepackage{epsf}
\usepackage{caption}
\usepackage{slashed}

\def\p{\partial}
\def\o{\overline}
\def\b{\bar}

\def\mc{\mathcal}
\def\mr{\mathrm}

\def\n{\nabla}

\theoremstyle{plain}
\newtheorem{theorem}{Theorem}[section]
\theoremstyle{plain}
\newtheorem{definition}[theorem]{Definition}
\theoremstyle{plain}

\theoremstyle{remark}
\newtheorem{remark}[theorem]{Remark}
\theoremstyle{remark}

\theoremstyle{remark}

\theoremstyle{plain}
\newtheorem{proposition}[theorem]{Proposition}
\numberwithin{equation}{section}
\theoremstyle{plain}

\theoremstyle{plain}
\newtheorem{corollary}[theorem]{Corollary}
\newtheorem{problem}[theorem]{Problem}

\usepackage{fancyhdr}
\makeatletter
\@namedef{subjclassname@2020}{\textup{2020} Mathematics Subject Classification}
\makeatother

\begin{document}

\title{K\"ahler metrics of negative  holomorphic (bi)sectional curvature on a compact relative K\"ahler fibration}

\author{Xueyuan Wan}
\address{Mathematical Science Research Center, Chongqing University of
Technology, Chongqing 400054, China}
\email{xwan@cqut.edu.cn}

\date{\today}

\begin{abstract}

For a compact relative K\"ahler fibration over a compact K\"ahler manifold with negative holomorphic sectional curvature, if the relative K\"ahler form on each fiber also exhibits negative holomorphic sectional curvature, we can construct K\"ahler metrics with negative holomorphic sectional curvature on the total space. Additionally, if this form induces a Griffiths negative Hermitian metric on the relative tangent bundle, and the base admits a K\"ahler metric with negative holomorphic bisectional curvature, we can also construct K\"ahler metrics with negative holomorphic bisectional curvature on the total space. As an application, for a non-trivial fibration where both the fibers and base have K\"ahler metrics with negative holomorphic bisectional curvature, and the fibers are one-dimensional, we can explicitly construct K\"ahler metrics of negative holomorphic bisectional curvature on the total space, thus resolving a question posed by To and Yeung for the case where the fibers have dimension one.

\end{abstract}


\maketitle


\section{Introduction}

Negative curvature has long been a significant focus of research in complex geometry. It is well known that any compact complex manifold with negative holomorphic sectional curvature must be Kobayashi hyperbolic. Moreover, when such a manifold possesses a fibration structure, studying its negative curvature properties becomes even more intriguing. In this paper, we investigate the curvature properties of such manifolds. Let \( p: \mathcal{X} \to \mathcal{B} \) be a compact holomorphic fibration, where \( p \) is a proper holomorphic submersion between two compact complex manifolds \( \mathcal{X} \) and \( \mathcal{B} \).

In complex geometry, the notions of holomorphic sectional curvature and holomorphic bisectional curvature are two fundamental concepts (see Definition \ref{defn-bisec}). A natural question arises: under what conditions does the total space \( \mathcal{X} \) admit Hermitian (or K\"ahler) metrics with negative holomorphic (bi)sectional curvature?

The simplest nontrivial compact holomorphic fibration is the Kodaira surface, defined as a smooth compact complex surface admitting a Kodaira fibration \cite{Kodaira+1975+1511+1519}. The study of negative curvature properties on Kodaira surfaces has been rich. Cheung \cite{MR990192} constructed a K\"ahler metric with negative holomorphic sectional curvature, while Tsai \cite{MR1019707} developed a Hermitian metric with negative holomorphic bisectional curvature. By defining a local immersion from a Kodaira surface into the Teichm\"uller space \( \mathcal{T}_{g,1} \), To and Yeung \cite{MR2820140} demonstrated the existence of a K\"ahler metric with negative holomorphic bisectional curvature on the Kodaira surface. Independently, Tsui \cite{Tsui_2006} constructed a K\"ahler metric with negative holomorphic bisectional curvature on certain special Kodaira surfaces using a completely different approach.

For general compact holomorphic fibrations, finding a K\"ahler metric with negative holomorphic bisectional curvature on a general compact holomorphic fibration is a challenging problem. In \cite[Theorem 1]{MR2820140}, To and Yeung successfully proved the existence of such K\"ahler metrics on the Kodaira surface. Subsequently, in \cite[Remark (i)]{MR2820140}, they posed a related question in the higher-dimensional setting:

\begin{problem}\label{pro1}
In general given a non-trivial fibration for which fibers and base are all equipped with K\"ahler metrics of negative holomorphic bisectional curvature, one may ask whether the total space of the fibration admits a K\"ahler metric with a similar curvature property. \cite[Theorem 1]{MR2820140} provides an affirmative example to such a problem.
\end{problem}

In response to this question, we have obtained the following result for relative K\"ahler fibrations, addressing the case where the relative tangent bundle is Griffiths negative (see Definition \ref{G-N} for Griffiths negative vector bundles). A relative K\"ahler fibration is a holomorphic fibration equipped with a relative K\"ahler form, that is, a smooth closed \((1,1)\)-form that is positive along each fiber (see Definition \ref{relative Kahler}).

\begin{theorem}\label{thm2}
Let \( p: \mathcal{X} \to \mathcal{B} \) be a compact relative K\"ahler fibration over a compact K\"ahler manifold \( \mathcal{B} \) with negative holomorphic bisectional curvature. Suppose the relative K\"ahler form induces a Griffiths negative Hermitian metric on the relative tangent bundle \( T_{\mathcal{X}/\mathcal{B}} \). Then there exist K\"ahler metrics on \( \mathcal{X} \) with negative holomorphic bisectional curvature.
\end{theorem}
\begin{remark}
In fact, the assumption that the relative tangent bundle is Griffiths negative already implies that the family is effectively parametrized. Moreover, if the total space is a product or a holomorphic fiber bundle, then it cannot admit any K\"ahler metric with negative holomorphic bisectional curvature; see \cite[Theorem]{MR436056}.
\end{remark}

In particular, when the fibre is one-dimensional, and the family is effectively parametrized (that is, the Kodaira-Spencer map is injective at every point), the relative tangent bundle $T_{\mathcal{X}/\mathcal{B}}$ is automatically Griffiths negative.  Indeed, by a result of Schumacher \cite[Main Theorem]{MR2969273}, the relative canonical bundle $K_{\mathcal{X}/\mathcal{B}}$ is then a positive line bundle, and hence its dual $T_{\mathcal{X}/\mathcal{B}}$ carries a Griffiths negative Hermitian metric.  Consequently,  by using Theorem \ref{thm2} in this setting, we can resolve Problem~\ref{pro1} by explicitly constructing a family of K\"ahler metrics exhibiting negative holomorphic bisectional curvature.
The specific result is as follows:

\begin{corollary}\label{cor1}
Let \( p: \mathcal{X} \to \mathcal{B} \) be a compact holomorphic fibration over a compact K\"ahler manifold \( \mathcal{B} \) with negative holomorphic bisectional curvature. Suppose the fibration is a holomorphic family of compact Riemann surfaces of genus \( \geq 2 \) and is effectively parametrized. Then there exist K\"ahler metrics on \( \mathcal{X} \) with negative holomorphic bisectional curvature.
\end{corollary}

\begin{remark}
For the Kodaira surface, defined as a smooth compact complex surface \( \mathcal{X} \) that admits a Kodaira fibration, there exists a connected fibration \( p: \mathcal{X} \rightarrow \mathcal{B} \) over a smooth compact Riemann surface \( \mathcal{B} \), where \( p \) is a proper holomorphic submersion and the associated Kodaira-Spencer map is injective at each point (see \cite{MR2820140}). By \cite[Application]{MR2969273}, the genus of \( \mathcal{B} \) is \( \geq 2 \), and so \( \mathcal{B} \) is a negatively curved Riemann surface. Based on the above corollary, we can explicitly construct a smooth family of K\"ahler metrics \( \Omega = k(p^* \omega_{\mathcal{B}}) + \omega_{\mathcal{X}} \) on the Kodaira surface with negative holomorphic bisectional curvature for sufficiently large \( k \). Our approach differs notably from the method in \cite{MR2820140}.
\end{remark}

On the other hand, if both the base and the fiber admit Hermitian metrics with negative holomorphic sectional curvature, Cheung \cite[Theorem 1]{MR990192} proved that the entire compact holomorphic fibration also admits a Hermitian metric with negative holomorphic sectional curvature. In \cite{MR4100006}, Chaturvedi and Heier demonstrated that Cheung's metric exhibits positive holomorphic sectional curvature when both the base and fibers have positive holomorphic sectional curvature.
 Consequently, it is natural to consider the existence of K\"ahler metrics with negative (resp. positive) holomorphic sectional curvature on such fibrations. Regarding this question, we have obtained the following result for compact relative K\"ahler fibrations. 
\begin{theorem}\label{thm:HSC}
Let \( p: \mathcal{X} \to \mathcal{B} \) be a compact relative K\"ahler fibration such that the restriction of the relative K\"ahler form to each fiber has negative (resp. positive) holomorphic sectional curvature. If the base manifold \( \mathcal{B} \) also admits a K\"ahler metric of negative (resp. positive) holomorphic sectional curvature, then there exist K\"ahler metrics on \( \mathcal{X} \) with negative (resp. positive) holomorphic sectional curvature.
\end{theorem}
\begin{remark}
Note that Cheung \cite[Theorem 1]{MR990192} also defined a Hermitian metric and proved that it has negative holomorphic sectional curvature. However, there is a fundamental difference between Cheung's metric \eqref{Cheung metric} and our metric \eqref{Kahler metrics}. Specifically, even in the case of a relative K\"ahler fibration, Cheung's metric need not be K\"ahler. For further details on the distinction between these two metrics, refer to Remark \ref{rem-difference}.
\end{remark}
\begin{remark}
The study of K\"ahler manifolds with positive holomorphic sectional curvature has produced a rich collection of examples. For instance, in~\cite{MR4033928}, Yang and Zheng generalized Hitchin's construction and proved that every Hirzebruch manifold $M_{n, k} = \mathbb{P}(H^k \oplus \mathcal{O}_{\mathbb{C P}^{n-1}})$ admits a K\"ahler metric of positive holomorphic sectional curvature in each of its K\"ahler classes.
\end{remark}

This article is organized as follows:
 In Section \ref{sec:preli}, we will review the definitions of Griffiths negative vector bundles and the negativity of holomorphic (bi)sectional curvature. Section \ref{sec:cur} focuses on constructing a family of K\"ahler metrics on the total space and analyzing its curvature properties. In this section, we will prove Theorems \ref{thm:HSC} and \ref{thm2}. Finally, in Section \ref{sec:app}, we will explore a special case of relative K\"ahler fibrations, specifically holomorphic families of canonically polarized manifolds, and we will prove Corollary \ref{cor1}.

\

\text{{\bfseries{Acknowledgments.}}} The author would like to thank Ya Deng, Xu Wang, and Bo Yang for many helpful discussions. The author is especially grateful to Bo Yang for drawing attention to~\cite{MR4100006, MR4852033} and for suggesting the consideration of the case of positive holomorphic sectional curvature.
 Xueyuan Wan is sponsored by the National Key R\&D Program of China (Grant No. 2024YFA1013200) and the Natural Science Foundation of Chongqing, China (Grant No. CSTB2024NSCQ-LZX0040, CSTB2023NSCQ-LZX0042).

\

\section{Griffiths negativity and holomorphic (bi)sectional curvature}\label{sec:preli}

This section will review the definitions of the Chern connection and its curvature for a Hermitian holomorphic vector bundle. One can refer to \cite[Chapter 1]{Kobayashi+1987} for more details. Throughout this paper, we will adopt the Einstein summation convention.

Let $\pi:(E,h^E)\to X$ be a Hermitian holomorphic vector bundle over a complex manifold $X$, where $\mr{rank} E=r$ and $\dim X=n$. The Chern connection of $(E,h^E)$, denoted by $\n^E$, preserves the metric $h^E$ and is of $(1,0)$-type. Given a local holomorphic frame $\{e_i\}_{1\leq i\leq r}$ of $E$, the connection satisfies
\[
  \n^E e_i = \theta^j_i e_j,
\]
where $\theta = (\theta^j_i)$ represents the connection form of $\n^E$. More precisely,
\[
  \theta^j_i = \partial h_{i\b{k}} h^{\b{k}j},
\]
where $h_{i\b{k}} := h(e_i, e_k)$. The Chern curvature of $(E, h^E)$ is denoted by $R^E = (\n^E)^2 \in A^{1,1}(X, \mr{End}(E))$, and can be written as
\[
  R^E = R^j_i e_j \otimes e^i \in A^{1,1}(X, \mr{End}(E)),
\]
where $R = (R^j_i)$ is the curvature matrix, whose entries are $(1,1)$-forms. Here, $\{e^i\}_{1\leq i\leq r}$ represents the dual frame of $\{e_i\}_{1\leq i\leq r}$. The curvature matrix $R = (R^j_i)$ is given by
\[
  R^j_i = d\theta^j_i + \theta^j_k \wedge \theta^k_i = \b{\partial} \theta^j_i.
\]

Let $\{z^\alpha\}_{1\leq \alpha \leq n}$ be local holomorphic coordinates of $X$. The components of the curvature can be expressed as
\[
  R^j_i = R^j_{i\alpha\b{\beta}} dz^\alpha \wedge d\b{z}^\beta,
\]
with
\[
  R_{i\b{j}} := R^k_i h_{k\b{j}} = R_{i\b{j}\alpha\b{\beta}} dz^\alpha \wedge d\b{z}^\beta.
\]
Thus, we have
\[
  R_{i\b{j}\alpha\b{\beta}} = R^k_{i\alpha\b{\beta}} h_{k\b{j}} = -\partial_\alpha \partial_{\b{\beta}} h_{i\b{j}} + h^{\b{l}k} \partial_\alpha h_{i\b{l}} \partial_{\b{\beta}} h_{k\b{j}},
\]
where $\partial_\alpha := \partial/\partial z^\alpha$ and $\partial_{\b{\beta}} := \partial/\partial\b{z}^\beta$.

Griffiths positivity and negativity are defined as follows:
\begin{definition}\label{G-N}
A Hermitian holomorphic vector bundle $\pi: (E,h^E) \to X$ is said to be Griffiths positive (resp. negative) if 
\[
   R_{i\b{j}\alpha\b{\beta}} v^i \b{v}^j \xi^\alpha \b{\xi}^\beta > 0 \quad (\text{resp.} < 0)
\]
for any non-zero elements $v = v^i e_i \in E|_z$ and $\xi = \xi^\alpha \frac{\partial}{\partial z^\alpha} \in T_X|_z$ at any point $z \in X$.
\end{definition}

In particular, when $E = T_X$, the holomorphic tangent bundle of $X$, the holomorphic (bi)sectional curvatures are defined as follows, see also \cite{MR227901}.
\begin{definition}\label{defn-bisec}
For any two non-zero $(1,0)$-type tangent vectors $\xi = \xi^\alpha \frac{\partial}{\partial z^\alpha}$ and $\eta = \eta^\alpha \frac{\partial}{\partial z^\alpha}$, the holomorphic sectional curvature along the direction $\xi$ is given by
\[
  H(\xi) := \frac{R_{\xi\b{\xi}\xi\b{\xi}}}{\|\xi\|^4} = \frac{R_{\gamma\b{\delta}\alpha\b{\beta}} \xi^\alpha \b{\xi}^\beta \xi^\gamma \b{\xi}^\delta}{(h_{\alpha\b{\beta}} \xi^\alpha \b{\xi}^\beta)^2}.
\]
The holomorphic bisectional curvature along $\xi$ and $\eta$ is
\[
  H(\xi, \eta) := \frac{R_{\xi\b{\xi}\eta\b{\eta}}}{\|\xi\|^2 \|\eta\|^2} = \frac{R_{\gamma\b{\delta}\alpha\b{\beta}}  \eta^\gamma \b{\eta}^\delta \xi^\alpha \b{\xi}^\beta}{(h_{\alpha\b{\beta}} \xi^\alpha \b{\xi}^\beta) (h_{\gamma\b{\delta}} \eta^\gamma \b{\eta}^\delta)}.
\]
The holomorphic sectional curvature is said to be negative (respectively, positive) if
\[
H(\xi)<0 \quad (\text{respectively, } H(\xi)>0)
\]
for every nonzero tangent vector $\xi$.  
Similarly, the holomorphic bisectional curvature is said to be negative (respectively, positive) if
\[
H(\xi,\eta)<0 \quad (\text{respectively, } H(\xi,\eta)>0)
\]
for all nonzero tangent vectors $\xi$ and $\eta$.
\end{definition}

\section{Curvature of K\"ahler metrics on a fibration}\label{sec:cur}

 This section will prove the negativity of holomorphic sectional curvature and holomorphic bisectional curvature for a compact relative K\"ahler fibration.
 
\subsection{Definition of K\"ahler metrics}

Let \( p: \mathcal{X} \to \mathcal{B} \) be a holomorphic fibration with compact fibers, that is, a proper holomorphic submersion between two compact complex manifolds \( \mathcal{X} \) and \( \mathcal{B} \). Let \( (z, v) = (z^1, \dots, z^m, v^1, \dots, v^n) \) denote the local holomorphic coordinates of the total space \( \mathcal{X} \), where \( p(z, v) = z \). Here, \( z = (z^\alpha)_{1 \leq \alpha \leq m} \), $m=\dim\mc{B}$, represents the local holomorphic coordinates on \( \mathcal{B} \), and \( v = (v^i)_{1 \leq i \leq n} \), $n=\dim X_z$ ($X_z:=p^{-1}(z)$), represents the local holomorphic coordinates on the fibers. Denote by $T_{\mc{X}}$ and $T_{\mc{B}}$ the holomorphic tangent bundle of $\mc{X}$ and $\mc{B}$, respectively. 

First, we will define a relative K\"ahler fibration. One can refer to \cite[Section 1]{MR4631050} for more details.
\begin{definition}\label{relative Kahler}
 We call a holomorphic fibration $p:\mathcal{X} \rightarrow \mathcal{B}$  a relative K\"ahler fibration if there exists a real, smooth, $d$-closed (1,1)-form $\omega_{\mc{X}}$ on $\mc{X}$ such that $\omega_{\mc{X}}|_{X_z}$ is positive for any $z\in \mc{B}$.
\end{definition}
\begin{remark}
	It is worth noting that Li~\cite[Theorem 1.1]{MR4852033} characterizes the existence of a K\"ahler metric on $\mathcal{X}$. This is equivalent to the base being a compact K\"ahler manifold and the existence of a class $[Q] \in H^2(\mathcal{X}, \mathbb{R})$ whose restriction to each fiber is a K\"ahler class.
\end{remark}

Now we assume that $p:(\mc{X},\omega_{\mc{X}})\to \mc{B}$ is a relative K\"ahler fibration. By $\b{\p}$-Poincar\'e Lemma, there exists a local weight, say $\phi$, such that
\begin{equation*}
  \omega_{\mc{X}}=\sqrt{-1}\p\b{\p}\phi.
\end{equation*}
By utilizing this relative K\"ahler form, we can obtain a canonical horizontal-vertical decomposition of the holomorphic tangent bundle of $\mc{X}$. See for example \cite[Section 1]{MR3955542}, such horizontal lifts (and related curvature computations) go back to Siu's work on the Weil--Petersson metric \cite{MR859202}.
With respect to \( \phi \), the canonical lift of \( \frac{\partial}{\partial z^\alpha} \) is given by
\[
\tfrac{\delta}{\delta z^\alpha} := \tfrac{\partial}{\partial z^\alpha} - \phi_{\alpha \bar{j}} \phi^{i \bar{j}} \tfrac{\partial}{\partial v^i}.
\]
Here  \( (\phi^{i \bar{j}}) \) denotes the inverse matrix of \( (\phi_{i \bar{j}}) \).
 Similarly, we define
\[
\delta v^i = dv^i + \phi^{i \bar{j}} \phi_{\bar{j} \alpha} dz^\alpha.
\]

It can be verified that
 \[ \mathcal{H} = \mathrm{Span}_{\mathbb{C}} \left\{ \tfrac{\delta}{\delta z^\alpha} \right\} \] forms a horizontal subbundle of \( T_\mathcal{X} \). The holomorphic vertical subbundle is \[ \mathcal{V} = \mathrm{Span}_{\mathbb{C}} \left\{ \tfrac{\partial}{\partial v^i} \right\} =\ker \{p_*:T_\mc{X}\to T_\mc{B}\}\] which corresponds to the relative tangent bundle \( T_{\mathcal{X}/\mathcal{B}} \). Equivalently, the horizontal subbundle can be defined intrinsically by
$$
\mc{H}:=\left\{\zeta \in T_\mc{X} \mid \omega_{\mc{X}}(\zeta, \bar{\nu})=0 \text { for all } \nu \in \mc{V}\right\}
$$
so that 
$$ T_\mc{X}=\mc{H}\oplus \mc{V}$$ globally and, in particular, $ \mathcal{H} $ and $\mc{V}$ are $\omega_\mc{X}$-orthogonal.

The local frame \( \left\{ \tfrac{\delta}{\delta z^\alpha}, \tfrac{\partial}{\partial v^i} \right\} \) of \( T_\mathcal{X} \) is dual to \( \{ dz^\alpha, \delta v^i \} \). Furthermore, we have the following decomposition
\[
\omega_{\mathcal{X}} = \sqrt{-1} \partial \bar{\partial} \phi = c(\phi) + \sqrt{-1} \phi_{i \bar{j}} \delta v^i \wedge \delta \bar{v}^j,
\]
where \( c(\phi) \) is the geodesic curvature form given by
\[
c(\phi) = \sqrt{-1} c(\phi)_{\alpha \bar{\beta}} dz^\alpha \wedge d\bar{z}^\beta, \quad c(\phi)_{\alpha \bar{\beta}} = \phi_{\alpha \bar{\beta}} - \phi_{\alpha \bar{j}} \phi^{\bar{j} i} \phi_{i \bar{\beta}}.
\]

Let \( \omega_{\mathcal{B}} = \sqrt{-1} \psi_{\alpha \bar{\beta}} dz^\alpha \wedge d\bar{z}^\beta \) be a K\"ahler metric on the base manifold $\mc{B}$. We define
\begin{equation}\label{Kahler metrics}
  \Omega = k(p^* \omega_{\mathcal{B}}) + \omega_{\mathcal{X}},
\end{equation}
which is a K\"ahler metric on the total space \( \mathcal{X} \) for large $k$.
 In terms of local coordinates, we have
\[
\Omega = \sqrt{-1} \Omega_{\alpha \bar{\beta}} dz^\alpha \wedge d\bar{z}^\beta + \sqrt{-1} \phi_{i \bar{j}} \delta v^i \wedge \delta \bar{v}^j,
\]
where
\[
\Omega_{\alpha \bar{\beta}} = k \psi_{\alpha \bar{\beta}} + c(\phi)_{\alpha \bar{\beta}}.
\]
With respect to the splitting $T_\mc{X}=\mc{H}\oplus \mc{V}$, the K\"ahler form $\Omega$ induces Hermitian metrics on $\mc{H}$ and $\mc{V}$, represented locally by the matrices $(\Omega_{\alpha\bar\beta})$ and $(\phi_{i\bar j})$, respectively.
\begin{remark}\label{rem-difference}
For a general proper holomorphic submersion \(\pi: X \to Y\), Cheung \cite[Section 5]{MR990192} defined a Hermitian metric \(\Psi_\lambda\) on the total space \(X\) by the expression
\begin{equation}\label{Cheung metric}
\Psi_\lambda = \lambda(\pi^*\omega_Y) + \Phi,
\end{equation}
and proved that \(\Psi_\lambda\) has negative holomorphic sectional curvature if both \(\omega_Y\) and \(\Phi|_{\text{fiber}}\) have negative holomorphic sectional curvature. Here, \(\Phi\) is a Hermitian \((1,1)\)-form on \(X\) defined by
\[
\Phi\left(Z_1, Z_2\right)(\mathbf{p}) \equiv \varphi_t\left(\operatorname{proj}_{\tilde{\mathbf{G}}} Z_1(\mathbf{p}), \operatorname{proj}_{\tilde{\mathbf{G}}} Z_2(\mathbf{p})\right),
\]
where \(\mathbf{p} \in \pi^{-1}(t)\) and \(\operatorname{proj}_{\tilde{\mathbf{G}}}\) is the projection onto the fiber direction with respect to some Hermitian metric \(\tilde{\mathbf{G}}\) on \(X\). It is important to note that \(\Phi\) is positive along the fiber direction and degenerates along the horizontal direction. Consequently, \(\Phi\) need not be \(d\)-closed, even in the case of a relative K\"ahler fibration. Therefore, \(\Psi_\lambda\) may not be K\"ahler, even in the context of a relative K\"ahler fibration. This highlights a key difference between the metric defined in \eqref{Cheung metric} and our metric defined in \eqref{Kahler metrics}.

\end{remark}

\subsection{Negativity/Positivity of holomorphic sectional curvature}

In this subsection, we will prove the K\"ahler metrics in \eqref{Kahler metrics} have negative/positive holomorphic sectional curvature and prove Theorem \ref{thm:HSC}.

Let  $\nabla = \nabla^{1,0} + \bar{\partial}$  denote the Chern connection associated with  $\Omega$, and define the Chern curvature as
\begin{equation*}
  R=\n^2=\n^{1,0}\circ \b{\p}+\b{\p}\circ \n^{1,0}\in A^{1,1}(\mc{X}, \mr{End}(T_{\mc{X}})).
\end{equation*}
Then we have the following expressions
\begin{align}\label{equ8}
\begin{split}
  \n^{1,0}\tfrac{\delta}{\delta z^\alpha}&=  \left\langle\n^{1,0}\tfrac{\delta}{\delta z^\alpha},\tfrac{\delta}{\delta z^\beta} \right\rangle\Omega^{\b{\beta}\gamma}\tfrac{\delta}{\delta z^\gamma}+ \left\langle\n^{1,0}\tfrac{\delta}{\delta z^\alpha},\tfrac{\p}{\p v^j}\right\rangle \phi^{\b{j}i}\tfrac{\p}{\p v^i}\\
&  =\p\Omega_{\alpha\b{\beta}}\Omega^{\b{\beta}\gamma}\tfrac{\delta}{\delta z^\gamma}
 \end{split}
\end{align}
and 
\begin{align}\label{equ9}
\begin{split}
  \n^{1,0}\tfrac{\p}{\p v^i}&=\left\langle \n^{1,0}\tfrac{\p}{\p v^i},\tfrac{\delta}{\delta z^\beta}\right\rangle \Omega^{\b{\beta}\alpha}\tfrac{\delta}{\delta z^\alpha}+\left\langle \n^{1,0}\tfrac{\p}{\p v^i},\tfrac{\p}{\p v^j}\right\rangle\phi^{\b{j}k}\tfrac{\p}{\p v^k}\\
  &=\p(\phi_{\b{\beta}k}\phi^{k\b{j}})\phi_{i\b{j}}\Omega^{\b{\beta}\alpha}\tfrac{\delta}{\delta z^\alpha}+\p\phi_{i\b{j}}\phi^{\b{j}k}\tfrac{\p}{\p v^k}.
 \end{split}
\end{align}
Here we used the $\Omega$-orthogonality of the decomposition $T_\mc{X}=\mc{H}\oplus \mc{V}$, namely $\bigl\langle \delta/\delta z^\alpha,\ \partial/\partial v^j\bigr\rangle=0$.

From \eqref{equ8} and \eqref{equ9}, the Chern curvature  $R$  of  $\Omega$  can be expressed as
\begin{align*}
\begin{split}
  R\tfrac{\delta}{\delta z^\alpha}&=(\n^{1,0}\circ \b{\p}+\b{\p}\circ \n^{1,0})(\tfrac{\delta}{\delta z^\alpha})\\
  &=\n^{1,0}(\b{\p}(-\phi_{\alpha\b{j}}\phi^{\b{j}i})\tfrac{\p}{\p v^i})+\b{\p}\left(\p\Omega_{\alpha\b{\beta}}\Omega^{\b{\beta}\gamma}\tfrac{\delta}{\delta z^\gamma}\right)\\
  &=\left[\b{\p}\left(\p\Omega_{\alpha\b{\beta}}\Omega^{\b{\beta}\gamma}\right)+\b{\p}(\phi_{\alpha\b{j}}\phi^{\b{j}i})\wedge \p(\phi_{\b{\beta}k}\phi^{k\b{l}})\phi_{i\b{l}}\Omega^{\b{\beta}\gamma}\right]\tfrac{\delta}{\delta z^\gamma}\\
  &\quad +\left[-\p\b{\p}(\phi_{\alpha\b{j}}\phi^{\b{j}i})+\b{\p}(\phi_{\alpha\b{j}}\phi^{\b{j}k})\wedge \p\phi_{k\b{j}}\phi^{\b{j}i}\right.\\
  &\quad\left.+\p\Omega_{\alpha\b{\beta}}\Omega^{\b{\beta}\gamma}\wedge \b{\p}(\phi_{\gamma\b{j}}\phi^{\b{j}i})\right]\tfrac{\p}{\p v^i}
 \end{split}
\end{align*}
and 
\begin{align*}
\begin{split}
  R\tfrac{\p}{\p v^i}&=\b{\p}\circ \n^{1,0}(\tfrac{\p}{\p v^i})\\
  &=\b{\p}\left[\p(\phi_{\b{\beta}k}\phi^{k\b{j}})\phi_{i\b{j}}\Omega^{\b{\beta}\alpha}\tfrac{\delta}{\delta z^\alpha}+\p\phi_{i\b{j}}\phi^{\b{j}k}\tfrac{\p}{\p v^k}\right]\\
  &=\b{\p}\left(\p(\phi_{\b{\beta}k}\phi^{k\b{j}})\phi_{i\b{j}}\Omega^{\b{\beta}\alpha}\right)\tfrac{\delta}{\delta z^\alpha}\\
  &\quad+\left[\b{\p}(\p\phi_{i\b{j}}\phi^{\b{j}k})+\p(\phi_{\b{\beta}s}\phi^{s\b{j}})\phi_{i\b{j}}\Omega^{\b{\beta}\alpha}\wedge\b{\p}(\phi_{\alpha\b{l}}\phi^{\b{l}k})\right]\tfrac{\p}{\p v^k}.
 \end{split}
\end{align*}
By taking the inner product of  $R \frac{\delta}{\delta z^\alpha}$  and  $R \frac{\partial}{\partial v^i} $ with  $\frac{\delta}{\delta z^\beta}$  and  $\frac{\partial}{\partial v^j}$, we obtain the following proposition.
\begin{proposition}\label{prop1}
	We have
	\begin{align}
  \left\langle  R\tfrac{\delta}{\delta z^\alpha},\tfrac{\delta}{\delta z^\beta}\right\rangle&=\left\langle R^{\mc{H}}\tfrac{\delta}{\delta z^\alpha},\tfrac{\delta}{\delta z^\beta}\right\rangle+\b{\p}(\phi_{\alpha\b{j}}\phi^{\b{j}i})\wedge \p(\phi_{\b{\beta}k}\phi^{k\b{l}})\phi_{i\b{l}}.\label{equ10}\\
  \left\langle R\tfrac{\p}{\p v^i},\tfrac{\p}{\p v^j}\right\rangle &=\left\langle R^\mc{V}\tfrac{\p}{\p v^i},\tfrac{\p}{\p v^j}\right\rangle+\p(\phi_{\b{\beta}s}\phi^{s\b{q}})\phi_{i\b{q}}\Omega^{\b{\beta}\alpha}\wedge\b{\p}(\phi_{\alpha\b{l}}\phi^{\b{l}k})\phi_{k\b{j}}.\label{equ11}\\
  \left\langle R\tfrac{\delta}{\delta z^\alpha},\tfrac{\p}{\p v^l}\right\rangle&=\left[-\p\b{\p}(\phi_{\alpha\b{j}}\phi^{\b{j}i})+\b{\p}(\phi_{\alpha\b{j}}\phi^{\b{j}k})\wedge \p\phi_{k\b{j}}\phi^{\b{j}i}\right.\label{equ12}\\
  &\left.\quad\,+\p\Omega_{\alpha\b{\beta}}\Omega^{\b{\beta}\gamma}\wedge \b{\p}(\phi_{\gamma\b{j}}\phi^{\b{j}i})\right]\phi_{i\b{l}}.\nonumber
\end{align}
Here, $R^\mc{H}$ denotes the Chern curvature of the Hermitian vector bundle $(\mc{H},(\Omega_{\alpha\b{\beta}}))$, and $R^\mc{V}$ is the Chern curvature of the Hermitian vector bundle $(\mc{V},(\phi_{i\b{j}}))$.
\end{proposition}

Recall that \( \Omega_{\alpha \bar{\beta}} = k \psi_{\alpha \bar{\beta}} + c(\phi)_{\alpha \bar{\beta}} \). We now derive the following estimates.

\begin{proposition}\label{prop2}
As \( k \to \infty \), the following estimates hold:
\begin{align}
    R_{\gamma \bar{\sigma} \alpha \bar{\beta}} &:= \left\langle R \left( \tfrac{\delta}{\delta z^\gamma}, \tfrac{\delta}{\delta \bar{z}^\sigma} \right) \tfrac{\delta}{\delta z^\alpha}, \tfrac{\delta}{\delta z^\beta} \right\rangle = R^{T_\mathcal{B}}_{\gamma \bar{\sigma} \alpha \bar{\beta}} k + O(1), \label{equ13} \\
    R_{\gamma \bar{\sigma} i \bar{j}} &:= \left\langle R \left( \tfrac{\delta}{\delta z^\gamma}, \tfrac{\delta}{\delta \bar{z}^\sigma} \right) \tfrac{\partial}{\partial v^i}, \tfrac{\partial}{\partial v^j} \right\rangle = R^{\mathcal{V}}_{\gamma \bar{\sigma} i \bar{j}} + O\left( \tfrac{1}{k} \right) = O(1), \label{equ14} \\
    R_{k \bar{l} i \bar{j}} &:= \left\langle R \left( \tfrac{\partial}{\partial v^k}, \tfrac{\partial}{\partial \bar{v}^l} \right) \tfrac{\partial}{\partial v^i}, \tfrac{\partial}{\partial v^j} \right\rangle = R^{\mathcal{V}}_{k \bar{l} i \bar{j}} + O\left( \tfrac{1}{k} \right), \label{equ15} \\
    R_{k \bar{\beta} i \bar{j}} &:= \left\langle R \left( \tfrac{\partial}{\partial v^k}, \tfrac{\delta}{\delta \bar{z}^\beta} \right) \tfrac{\partial}{\partial v^i}, \tfrac{\partial}{\partial v^j} \right\rangle = O(1), \label{equ16} \\
    R_{\alpha \bar{l} i \bar{j}} &:= \left\langle R \left( \tfrac{\delta}{\delta z^\alpha}, \tfrac{\partial}{\partial \bar{v}^l} \right) \tfrac{\partial}{\partial v^i}, \tfrac{\partial}{\partial v^j} \right\rangle = O(1), \label{equ17} \\
    R_{\alpha \bar{l} \gamma \bar{j}} &:= \left\langle R \left( \tfrac{\delta}{\delta z^\alpha}, \tfrac{\partial}{\partial \bar{v}^l} \right) \tfrac{\delta}{\delta z^\gamma}, \tfrac{\partial}{\partial v^j} \right\rangle = O(1), \label{equ18} \\
    R_{i \bar{\sigma} \alpha \bar{\beta}} &:= \left\langle R \left( \tfrac{\partial}{\partial v^i}, \tfrac{\delta}{\delta \bar{z}^\sigma} \right) \tfrac{\delta}{\delta z^\alpha}, \tfrac{\delta}{\delta z^\beta} \right\rangle = O(1). \label{equ19}
\end{align}
Here, \( f(k) = O(g(k)) \) as \( k \to \infty \) means that there exist constants \( k_0 > 0 \) and \( C > 0 \) such that \( |f(k)| < C g(k) \) for any \( k \geq k_0 \).
\end{proposition}

\begin{proof}
First, note that \( \frac{1}{k} \Omega_{\alpha \bar{\beta}} = \psi_{\alpha \bar{\beta}} + \frac{1}{k} c(\phi)_{\alpha \bar{\beta}} \), so \( k \Omega^{\alpha \bar{\beta}} = \psi^{\alpha \bar{\beta}} + O\left( \frac{1}{k} \right) \). Using Proposition \ref{prop1}  \eqref{equ10}, we have
\begin{align*}
   \left\langle R \tfrac{\delta}{\delta z^\alpha}, \tfrac{\delta}{\delta z^\beta} \right\rangle &= \bar{\partial} \partial \Omega_{\alpha \bar{\beta}} + \partial \Omega_{\alpha \bar{\sigma}} \wedge \bar{\partial} \Omega_{\gamma \bar{\beta}} \Omega^{\bar{\sigma} \gamma}+O(1) \\
   &= k \bar{\partial} \partial \psi_{\alpha \bar{\beta}} + O(1) \\
   &\quad + \left( k \partial \psi_{\alpha \bar{\sigma}} + O(1) \right) \wedge \left( k \bar{\partial} \psi_{\gamma \bar{\beta}} + O(1) \right) \left( \tfrac{1}{k} \psi^{\bar{\sigma} \gamma} + O\left( \tfrac{1}{k^2} \right) \right) \\
   &= k \left( \bar{\partial} \partial \psi_{\alpha \bar{\beta}} + \partial \psi_{\alpha \bar{\sigma}} \wedge \bar{\partial} \psi_{\gamma \bar{\beta}} \psi^{\bar{\sigma} \gamma} \right) + O(1) \\
   &= k \left\langle R^{T_\mathcal{B}} \tfrac{\delta}{\delta z^\alpha}, \tfrac{\delta}{\delta z^\beta} \right\rangle + O(1).
\end{align*}
Thus, \eqref{equ13} and \eqref{equ19} are proved. Now, using Proposition \ref{prop1} \eqref{equ11}, we get
\begin{align}\label{equ25}
    \left\langle R \tfrac{\partial}{\partial v^i}, \tfrac{\partial}{\partial v^j} \right\rangle &= \left\langle R^\mathcal{V} \tfrac{\partial}{\partial v^i}, \tfrac{\partial}{\partial v^j} \right\rangle + O\left( \tfrac{1}{k} \right) = O(1),
\end{align}
which proves \eqref{equ14} through \eqref{equ17}. Finally, Proposition \ref{prop1} \eqref{equ12} gives the proof of \eqref{equ18}.
\end{proof}

Next, we prove Theorem \ref{thm:HSC}. 
\begin{proof}[Proof of Theorem \ref{thm:HSC}]
For any vector \( X \) of type \( (1,0) \), we decompose it as follows
\[
X = Y + Z, \quad Y = a^\alpha \tfrac{\delta}{\delta z^\alpha}, \quad Z = b^i \tfrac{\partial}{\partial v^i}.
\]
Using Proposition \ref{prop2}, the holomorphic sectional curvature along the direction \( X \) is given by
\begin{align}\label{equ20}
\begin{split}
    R_{X \bar{X} X \bar{X}} &= \left\langle R(Y + Z, \bar{Y} + \bar{Z})(Y + Z), Y + Z \right\rangle \\
    &= R_{Y \bar{Y} Y \bar{Y}} + 2R_{Y \bar{Y} Y \bar{Z}} + 2R_{Y \bar{Y} Z \bar{Y}} + 4R_{Y \bar{Y} Z \bar{Z}} + R_{Y \bar{Z} Y \bar{Z}} \\
    &\quad + R_{Z \bar{Y} Z \bar{Y}} + 2R_{Y \bar{Z} Z \bar{Z}} + 2R_{Z \bar{Z} Z \bar{Y}} + R_{Z \bar{Z} Z \bar{Z}} \\
    &= R^{T_\mathcal{B}}_{p_* Y \o{p_* Y} p_* Y \o{p_* Y}} \cdot k + R^{\mathcal{V}}_{Z \bar{Z} Z \bar{Z}} + O\left( \tfrac{1}{k} \right) \|b\|^4 \\
    &\quad + O(1)(\|a\|^4 + \|a\|^3 \|b\| + \|a\|^2 \|b\|^2 + \|a\| \|b\|^3),
\end{split}
\end{align}
where \( a = (a^1, \dots, a^{m}) \) and \( b = (b^1, \dots, b^{n}) \),
and the norms of $a$ and $b$ are defined as follows
\begin{equation}\label{norms}
  \|a\|^2:={a^\alpha\o{a^\beta}\psi_{\alpha\b{\beta}}}=\|p_*Y\|^2_{\omega_\mc{B}}\text{ and }\|b\|^2:={b^i\o{b^j}\phi_{i\b{j}}}=\|Z\|^2_{\omega_{\mc{X}}}.
\end{equation}

By assumption, if both the base and the fibers have negative holomorphic sectional curvature, then
\begin{equation}\label{equ4}
     R^{T_\mc{B}}_{p_*Y\o{p_*Y}p_*Y\o{p_*Y}}\leq -\epsilon_0\|p_*Y\|_{\omega_\mc{B}}^4=-\epsilon_0|a^\alpha\o{a^\beta}\psi_{\alpha\b{\beta}}|^2= -\epsilon_0 \|a\|^4,
\end{equation}
for some small positive constant \( \epsilon_0 \), where \( -\epsilon_0 \) is taken to be the maximum holomorphic sectional curvature of \( (\mathcal{B}, \omega_{\mathcal{B}}) \). Similarly, we have
\begin{equation}\label{equ21}
    R^{\mathcal{V}}_{Z \bar{Z} Z \bar{Z}} \leq -\epsilon_1 \|b\|^4
\end{equation}
for some small positive constant \( \epsilon_1 \). Substituting \eqref{equ4} and \eqref{equ21} into \eqref{equ20} yields
\begin{align}\label{equ22}
\begin{split}
    R_{X \bar{X} X \bar{X}} &\leq -\epsilon_0 k \|a\|^4 - \epsilon_1 \|b\|^4 + O\left( \tfrac{1}{k} \right) \|b\|^4 \\
    &\quad + O(1)(\|a\|^4 + \|a\|^3 \|b\| + \|a\|^2 \|b\|^2 + \|a\| \|b\|^3).
\end{split}
\end{align}
Using Young's inequality, we have the following estimates:
\begin{align}\label{Yang-ineqn}
\begin{split}
   \|a\|^3\|b\|&\leq \tfrac{3}{4}k^{\tfrac{1}{6}}\|a\|^4+\tfrac{1}{4}\tfrac{1}{\sqrt{k}}\|b\|^4.\\
    \|b\|^3\|a\|&\leq \tfrac{3}{4}k^{-\tfrac{1}{6}}\|b\|^4+\tfrac{1}{4}\sqrt{k}\|a\|^4.\\
    \|a\|^2\|b\|^2&\leq \tfrac{\sqrt{k}}{2}\|a\|^4+\tfrac{1}{2\sqrt{k}}\|b\|^4.
 \end{split}
\end{align}
Thus, \eqref{equ22} simplifies to
\begin{align*}
\begin{split}
   R_{X\b{X}X\b{X}}&\leq \|a\|^4(-\epsilon_0 k+O(1)(1+\tfrac{3}{4}k^{\tfrac{1}{6}}+\tfrac{3}{4}\sqrt{k}))\\
   &\quad +\|b\|^4(-\epsilon_1+O(\tfrac{1}{k})+O(1)(\tfrac{1}{4\sqrt{k}}+\tfrac{3}{4}k^{-\tfrac{1}{6}}+\tfrac{1}{2\sqrt{k}}))\\
   &\leq \|a\|^4(-\epsilon_0 k+O(\sqrt{k}))+\|b\|^4(-\epsilon_1+O(k^{-\tfrac{1}{6}})).
 \end{split}
\end{align*}
Hence, there exists some \( k_0 > 0 \) such that for any \( k \geq k_0 \),
\[
-\epsilon_0 k + O(\sqrt{k}) < 0 \quad \text{and} \quad -\epsilon_1 + O( k^{-\tfrac{1}{6}}) < 0.
\]
Therefore,
\[
R_{X \bar{X} X \bar{X}} < 0,
\]
for any \( \|a\| \neq 0 \) or \( \|b\| \neq 0 \), i.e., for any non-zero vector \( X \). This concludes the proof of the negativity of the holomorphic sectional curvature of the K\"ahler metric \( \Omega = k p^*\omega_{\mathcal{B}} + \omega_{\mathcal{X}} \) for any \( k \geq k_0 \).

If both the base and the fibers have positive holomorphic sectional curvature, then
\begin{equation}\label{equ4}
     R^{T_\mc{B}}_{p_*Y\o{p_*Y}p_*Y\o{p_*Y}}\geq \epsilon_0\|p_*Y\|_{\omega_\mc{B}}^4=\epsilon_0|a^\alpha\o{a^\beta}\psi_{\alpha\b{\beta}}|^2= \epsilon_0 \|a\|^4,
\end{equation}
for some small positive constant \( \epsilon_0 \), where \( \epsilon_0 \) is taken to be the minimum holomorphic sectional curvature of \( (\mathcal{B}, \omega_{\mathcal{B}}) \). Similarly, we have
\begin{equation}\label{equ21}
    R^{\mathcal{V}}_{Z \bar{Z} Z \bar{Z}} \geq \epsilon_1 \|b\|^4
\end{equation}
for some small positive constant \( \epsilon_1 \). Similarly, one has
\begin{align}\label{equ22}
\begin{split}
    R_{X \bar{X} X \bar{X}} &\geq \epsilon_0 k \|a\|^4 + \epsilon_1 \|b\|^4 + O\left( \tfrac{1}{k} \right) \|b\|^4 \\
    &\quad + O(1)(\|a\|^4 + \|a\|^3 \|b\| + \|a\|^2 \|b\|^2 + \|a\| \|b\|^3).
\end{split}
\end{align}
Using Yang inequality \eqref{Yang-ineqn}, we obtain that
\begin{align*}
\begin{split}
   R_{X\b{X}X\b{X}}&\geq \|a\|^4(\epsilon_0 k-O(1)(1+\tfrac{3}{4}k^{\tfrac{1}{6}}+\tfrac{3}{4}\sqrt{k}))\\
   &\quad +\|b\|^4(\epsilon_1+O(\tfrac{1}{k})-O(1)(\tfrac{1}{4\sqrt{k}}+\tfrac{3}{4}k^{-\tfrac{1}{6}}+\tfrac{1}{2\sqrt{k}})).
 \end{split}
\end{align*}
Then there exists $k_0\geq 0$ such that for any $k\geq k_0$, one has
\begin{equation*}
  \epsilon_0 k-O(1)(1+\tfrac{3}{4}k^{\tfrac{1}{6}}+\tfrac{3}{4}\sqrt{k})\geq \frac{\epsilon_0}{2}k
\end{equation*}
and 
\begin{equation*}
  \epsilon_1+O(\tfrac{1}{k})-O(1)(\tfrac{1}{4\sqrt{k}}+\tfrac{3}{4}k^{-\tfrac{1}{6}}+\tfrac{1}{2\sqrt{k}})\geq \frac{\epsilon_1}{2}.
\end{equation*}
This follows that 
\begin{equation*}
  R_{X\b{X}X\b{X}}\geq \frac{\epsilon_0}{2}k\|a\|^4+\frac{\epsilon_1}{2}\|b\|^4
\end{equation*}
 Therefore,
\[
R_{X \bar{X} X \bar{X}} > 0,
\]
for any \( \|a\| \neq 0 \) or \( \|b\| \neq 0 \), i.e., for any non-zero vector \( X \). This concludes the proof of the positivity of the holomorphic sectional curvature of the K\"ahler metric \( \Omega = k p^*\omega_{\mathcal{B}} + \omega_{\mathcal{X}} \) for any large $k\geq k_0$. 
\end{proof}
\begin{remark}
By \cite{MR3489705} and \cite{MR3715350}, if each fiber has a K\"ahler metric with negative holomorphic sectional curvature, then the canonical bundle of each fiber is ample. Hence, this relative K\"ahler fibration is a holomorphic family of compact, canonically polarized manifolds.
\end{remark}

\subsection{Negativity of holomorphic bisectional curvature}

In this section, we will prove that the K\"ahler metrics in \eqref{Kahler metrics} have negative holomorphic bisectional curvature and prove Theorem \ref{thm2}.

Recall that the relative K\"ahler form \( \omega_{\mathcal{X}} \) is a real \((1,1)\)-form on \( \mathcal{X} \), which is positive when restricted to each fiber. We express it as follows
\[
\omega_{\mathcal{X}} = \sqrt{-1} \partial \bar{\partial} \phi = c(\phi) + \sqrt{-1} \phi_{i \bar{j}} \delta v^i \wedge \delta \bar{v}^j.
\]
The relative tangent bundle $T_{\mc{X}/\mc{B}}=\mc{V}$ is spanned by $\{\frac{\p}{\p v^i}\}_{1\leq i\leq n}$. Hence the K\"ahler form $\omega_{\mc{X}}$ induces a canonical Hermitian metric on $\mc{V}$ by 
\begin{equation*}
  \left\langle \tfrac{\p}{\p v^i},\tfrac{\p}{\p v^j}\right\rangle:=\phi_{i\b{j}}.
\end{equation*}

In this section, we assume that the induced Hermitian metric has Griffiths negative curvature. Denote
\begin{equation*}
  \Omega_0:=p^* \omega_{\mathcal{B}} +  \sqrt{-1} \phi_{i \bar{j}} \delta v^i \wedge \delta \bar{v}^j,
\end{equation*}
which is a Hermitian metric on $\mc{X}$.
Since \( \mathcal{X} \) is compact, there exist two uniform positive constants \( c_0 \) and \( C_0 \) such that
\begin{equation}\label{equ5}
c_0 \|X\|_{\Omega_0}^2 \|V\|^2 \leq -R^{\mathcal{V}}_{X \bar{X} V \bar{V}} \leq C_0 \|X\|_{\Omega_0}^2 \|V\|^2
\end{equation}
for any tangent vector $X$ and $V\in \mc{V}$.

For any non-zero vector \( X \) of type \( (1,0) \), we can decompose it as \[ X = Y + Z \] where \( Y = a^\alpha \tfrac{\delta}{\delta z^\alpha} \) and \( Z = b^i \tfrac{\partial}{\partial v^i} \). 

For fixed non-zero $Y,Z,V$, set
\[
P(t):=R^{\mc{V}}_{\,Y+tZ\,\overline{Y+tZ}\,V\bar V}
=t^{2}R^{\mc{V}}_{\,ZZ\bar V V}+t\bigl(R^{\mc{V}}_{\,YZ\bar V V}+R^{\mc{V}}_{\,ZY\bar V V}\bigr)+R^{\mc{V}}_{\,YY\bar V V}.
\]
Since $(V,\phi_{i\bar j})$ is Griffiths negative, we have $P(t)<0$ for all $t\in\mathbb R$, and therefore the discriminant is negative, that is,
\[
\bigl|R^{\mc{V}}_{\,YZ\bar V V}+R^{\mc{V}}_{\,ZY\bar V V}\bigr|^{2}<4\,R^{\mc{V}}_{\,ZZ\bar V V}\,R^{\mc{V}}_{\,YY\bar V V}.
\]
Equivalently,
\[
2\sqrt{R^{\mc{V}}_{\,ZZ\bar V V}\,R^{\mc{V}}_{\,YY\bar V V}}-\bigl|R^{\mc{V}}_{\,YZ\bar V V}+R^{\mc{V}}_{\,ZY\bar V V}\bigr|>0,
\]
where the square root is well-defined since $R^{\mc{V}}_{\,ZZ\bar V V}<0$ and $R^{\mc{V}}_{\,YY\bar V V}<0$.
By restricting to $\|Y\|_{\Omega_0}=\|Z\|_{\Omega_0}=\|V\|=1$, the left-hand side defines a continuous positive function on a compact set, hence it is bounded below by some $c_1>0$.
By homogeneity, we obtain
\begin{equation}\label{equ6}
2\sqrt{R^{\mc{V}}_{\,ZZ\bar V V}\,R^{\mc{V}}_{\,YY\bar V V}}-\bigl|R^{\mc{V}}_{\,YZ\bar V V}+R^{\mc{V}}_{\,ZY\bar V V}\bigr|
\ge c_1\,\|V\|^{2}\,\|Z\|_{\Omega_0}\,\|Y\|_{\Omega_0}.
\end{equation}

On the other hand, we also have
\[
\sqrt{R^{\mathcal{V}}_{Z \bar{Z} V \bar{V}} R^{\mathcal{V}}_{Y \bar{Y} V \bar{V}}} \leq C_0 \|V\|^2 \|Z\|_{\Omega_0} \|Y\|_{\Omega_0}.
\]
Combining this with \eqref{equ6}, we obtain
\begin{align*}
& \quad \left( -R^{\mathcal{V}}_{Z \bar{Z} V \bar{V}} - R^{\mathcal{V}}_{Y \bar{Y} V \bar{V}} \right) \left( 1 - \tfrac{c_1}{2C_0} \right) - \left| R^{\mathcal{V}}_{Y \bar{Z} V \bar{V}} + R^{\mathcal{V}}_{Z \bar{Y} V \bar{V}} \right| \\
&\geq 2 \sqrt{R^{\mathcal{V}}_{Z \bar{Z} V \bar{V}} R^{\mathcal{V}}_{Y \bar{Y} V \bar{V}}} \left( 1 - \tfrac{c_1}{2C_0} \right) - \left| R^{\mathcal{V}}_{Y \bar{Z} V \bar{V}} + R^{\mathcal{V}}_{Z \bar{Y} V \bar{V}} \right| \\&\geq 0.
\end{align*}
Thus, we conclude
\begin{align}\label{equ7}
\begin{split}
&\quad -R^{\mathcal{V}}_{Z \bar{Z} V \bar{V}} - R^{\mathcal{V}}_{Y \bar{Y} V \bar{V}} - \left| R^{\mathcal{V}}_{Y \bar{Z} V \bar{V}} + R^{\mathcal{V}}_{Z \bar{Y} V \bar{V}} \right| \\
&\geq \epsilon_0 \left( -R^{\mathcal{V}}_{Z \bar{Z} V \bar{V}} - R^{\mathcal{V}}_{Y \bar{Y} V \bar{V}} \right),
\end{split}
\end{align}
where \( \epsilon_0 = \frac{c_1}{2C_0} > 0 \).

For any two non-zero vectors \( X \) and \( W \) of type \( (1,0) \), we decompose them as follows
\[
X = Y + Z\text{ and }W = U + V,
\]
where \( Y = a^\alpha \frac{\delta}{\delta z^\alpha} \), \( Z = b^i \frac{\partial}{\partial v^i} \), \( U = c^\alpha \frac{\delta}{\delta z^\alpha} \), and \( V = d^i \frac{\partial}{\partial v^i} \). Then, the holomorphic bisectional curvature satisfies
\begin{align}\label{equ24}
\begin{split}
  R_{X \bar{X} W \bar{W}} &= \left\langle R(Y+Z, \o{Y+Z})(U+V), U+V \right\rangle \\
  &= R_{Y \bar{Y} U \bar{U}} + R_{Y \bar{Y} U \bar{V}} + R_{Y \bar{Y} V \bar{U}} + R_{Y \bar{Y} V \bar{V}} \\
  &\quad + R_{Y \bar{Z} U \bar{U}} + R_{Y \bar{Z} U \bar{V}} + R_{Y \bar{Z} V \bar{U}} + R_{Y \bar{Z} V \bar{V}} \\
  &\quad + R_{Z \bar{Y} U \bar{U}} + R_{Z \bar{Y} U \bar{V}} + R_{Z \bar{Y} V \bar{U}} + R_{Z \bar{Y} V \bar{V}} \\
  &\quad + R_{Z \bar{Z} U \bar{U}} + R_{Z \bar{Z} U \bar{V}} + R_{Z \bar{Z} V \bar{U}} + R_{Z \bar{Z} V \bar{V}}.
  \end{split}
\end{align}

By Proposition \ref{prop2} and equation \eqref{equ5}, and following the same reasoning as in the proof of \eqref{equ4} and by assumption $(\mc{B},\omega_{\mc{B}})$ has negative holomorphic bisectional curvature,  there exists a small positive constant \( \epsilon_1 \) such that
\begin{align*}
  R_{Y \bar{Y} U \bar{U}} &\leq \|a\|^2 \|c\|^2 (-\epsilon_1 k + O(1)), \\
  R_{Z \bar{Z} V \bar{V}} &\leq \|b\|^2 \|d\|^2 (-\epsilon_1 + O(\tfrac{1}{k})), \\
  R_{Y \bar{Y} V \bar{V}} &\leq \|a\|^2 \|d\|^2 (-\epsilon_1 + O(\tfrac{1}{k})), \\
  R_{Z \bar{Z} U \bar{U}} &\leq \|b\|^2 \|c\|^2 (-\epsilon_1 + O(\tfrac{1}{k})),
\end{align*}
where the norms of $\|a\|,\|b\|,\|c\|,\|d\|$ are defined by  \eqref{norms}.

The cross terms involving different combinations of \( Y \), \( Z \), \( U \), and \( V \) have the following bounds
\begin{align*}
  |R_{Y \bar{Y} U \bar{V}}| = |R_{Y \bar{Y} V \bar{U}}| &\leq \|a\|^2 \|c\| \|d\| O(1), \\
  |R_{Y \bar{Z} U \bar{U}}| = |R_{Z \bar{Y} U \bar{U}}| &\leq \|a\| \|b\| \|c\|^2 O(1), \\
  |R_{Y \bar{Z} V \bar{V}}| = |R_{Z \bar{Y} V \bar{V}}| &\leq \|a\| \|b\| \|d\|^2 O(1), \\
  |R_{Z \bar{Z} U \bar{V}}| = |R_{Z \bar{Z} V \bar{U}}| &\leq \|b\|^2 \|c\| \|d\| O(1), \\
  |R_{Y \bar{Z} U \bar{V}}| = |R_{Z \bar{Y} V \bar{U}}| &\leq \|a\| \|b\| \|c\| \|d\| O(1).
\end{align*}

Using equations \eqref{equ25} and \eqref{equ7} and the symmetry of the curvature tensor, we obtain
\begin{align}\label{equ23}
\begin{split}
&  \quad R_{Z\b{Z}U\b{U}}+R_{Z\b{Z}V\b{V}}+|R_{Z\b{Z}U\b{V}}+R_{Z\b{Z}V\b{U}}|\\
&=R_{U\b{U}Z\b{Z}}+R_{V\b{V}Z\b{Z}}+|R_{U\b{V}Z\b{Z}}+R_{V\b{U}Z\b{Z}}|\\
&\leq R_{U\b{U}Z\b{Z}}^{\mc{V}}+R_{V\b{V}Z\b{Z}}^{\mc{V}}+|R^\mc{V}_{U\b{V}Z\b{Z}}+R^{\mc{V}}_{V\b{U}Z\b{Z}}|\\
&\quad+O(\tfrac{1}{k})\|b\|^2(\|c\|+\|d\|)^2\\
&\leq \epsilon_0\left(R_{U\b{U}Z\b{Z}}^{\mc{V}}+R_{V\b{V}Z\b{Z}}^{\mc{V}}\right)+O(\tfrac{1}{k})\|b\|^2(\|c\|+\|d\|)^2\\
&\leq (-\epsilon_0\epsilon_1+O(\tfrac{1}{k}))\|b\|^2(\|c\|^2+\|d\|^2).
 \end{split}
\end{align}
Similarly, we have
\begin{align*}
\begin{split}
 &\quad  R_{Y\b{Y}V\b{V}}+R_{Z\b{Z}V\b{V}}+|R_{Z\b{Y}V\b{V}}+R_{Z\b{Y}V\b{V}}|\\
 &  \leq (-\epsilon_0\epsilon_1+O(\tfrac{1}{k}))\|d\|^2(\|a\|^2+\|b\|^2).
 \end{split}
\end{align*}

Now we begin to prove Theorem \ref{thm2}.

\begin{proof}[Proof of Theorem \ref{thm2}]

We will discuss the negativity of holomorphic bisectional curvature in the following five cases.

{\bf Case I:} If $\|b\|=0$, i.e. $Z=0$, then 
  \begin{align*}
\begin{split}
 R_{X\b{X}W\b{W}} &=R_{Y\b{Y}U\b{U}}+R_{Y\b{Y}U\b{V}}+R_{Y\b{Y}V\b{U}}+R_{Y\b{Y}V\b{V}}\\
 &\leq\|a\|^2\|c\|^2(-\epsilon_1k+O(1))+2\|a\|^2\|c\|\|d\|O(1)\\
 &\quad +\|a\|^2\|d\|^2(-\epsilon_1 +O(\tfrac{1}{k}))\\
 &\leq\|a\|^2\|c\|^2(-\epsilon_1k+O(1)+\sqrt{k}O(1))\\
 &\quad+\|a\|^2\|d\|^2(-\epsilon_1 +O(\tfrac{1}{k})+\tfrac{1}{\sqrt{k}}O(1)),
 \end{split}
\end{align*}
where the last inequality by the inequality $2\|c\|\|d\|\leq \|c\|^2\sqrt{k}+\|d\|^2/\sqrt{k}$.
Hence there exists a large $k_0\geq 0$ such that 
\begin{equation*}
  -\epsilon_1k+O(1)+\sqrt{k}O(1)<0 \text{ and } -\epsilon_1 +O(\tfrac{1}{k})+\tfrac{1}{\sqrt{k}}<0
\end{equation*}
for any $k\geq k_0$. Hence $ R_{X\b{X}W\b{W}}<0$.  

{\bf Case II:} If $\|d\|=0$, i.e. $V=0$, this reduces to the previous case since $ R_{X\b{X}W\b{W}}= R_{W\b{W}X\b{X}}$. Hence $ R_{X\b{X}W\b{W}}<0$ for any $k\geq k_0$.

{\bf Case III:} If $\|b\|=\|d\|=1$ and $\|a\|\leq \|c\|$, using \eqref{equ23}, \eqref{equ24} and the estimates on curvature tensors,  then 
\begin{align*}
\begin{split}
R_{X\b{X}W\b{W}}
&\leq\|a\|^2\|c\|^2(-\epsilon_1k+O(1))+2\|a\|^2\|c\|O(1)+2\|a\|\|c\|^2O(1)\\
&\quad +4\|a\|\|c\|O(1)+ \|a\|^2(-\epsilon_1 +O(\tfrac{1}{k}))+2\|a\|O(1)\\
&\quad +(-\epsilon_0\epsilon_1+O(\tfrac{1}{k}))(\|c\|^2+1)\\
&\leq \|a\|^2\|c\|^2(-\epsilon_1k+O(1)+4\sqrt{k}O(1))\\
&\quad+ \|a\|^2(-\epsilon_1 +O(\tfrac{1}{k})+O(\tfrac{1}{\sqrt{k}}))\\
&\quad+\|c\|^2(-\epsilon_0\epsilon_1+O(\tfrac{1}{k})+O(\tfrac{1}{\sqrt{k}}))+(-\epsilon_0\epsilon_1+O(\tfrac{1}{k})\\
&\quad+O(\tfrac{1}{\sqrt{k}}))+2\|a\|O(1).
 \end{split}
\end{align*}
By choosing  $k$  large enough such that
\begin{equation*}
   \begin{cases}
-\epsilon_1k+O(1)+4\sqrt{k}O(1) 	& <-\tfrac{\epsilon_1}{2}k,\\
 -\epsilon_1 +O(\tfrac{1}{k})+O(\tfrac{1}{\sqrt{k}})	&<0,\\
 -\epsilon_0\epsilon_1+O(\tfrac{1}{k})+O(\tfrac{1}{\sqrt{k}})&<0,\\
 -\epsilon_0\epsilon_1+O(\tfrac{1}{k})+O(\tfrac{1}{\sqrt{k}})&<-\tfrac{1}{2}\epsilon_0\epsilon_1.
 \end{cases}
\end{equation*}
We obtain
\begin{align*}
\begin{split}
  R_{X\b{X}W\b{W}}&\leq -\tfrac{\epsilon_1}{2}k\|a\|^4-\tfrac{1}{2}\epsilon_0\epsilon_1+2\|a\|O(1)\\
  &\leq (-\tfrac{\epsilon}{2}k+\tfrac{\sqrt{k}}{2}O(1))\|a\|^4+(-\tfrac{1}{2}\epsilon_0\epsilon_1+\tfrac{3}{2}k^{-\tfrac{1}{6}}O(1)),
 \end{split}
\end{align*}
where the last inequality by Young inequality $\|a\|\leq \frac{\|a\|^4\sqrt{k}}{4}+\frac{3}{4}k^{-\tfrac{1}{6}}$.  By taking $k$ sufficiently large, one has
\begin{equation*}
 \begin{cases}
 	  -\tfrac{\epsilon_1}{2}k+\tfrac{\sqrt{k}}{2}O(1)&<0, \\
 -\tfrac{1}{2}\epsilon_0\epsilon_1+\tfrac{3}{2}k^{-\tfrac{1}{6}}O(1)	&<0.
 \end{cases}
\end{equation*}
Thus, we can find $k_1>0$, such that $ R_{X\b{X}W\b{W}}<0$ for any $k\geq k_1$.

{\bf Case IV:} For the case where $\|b\|=\|d\|=1$ and $\|a\|\geq  \|c\|$,  this follows from Case III, as 
\begin{equation*}
 R_{X\b{X}W\b{W}}= R_{W\b{W}X\b{X}}<0
\end{equation*}
for any $k\geq k_1$.

{\bf Case V:} If $\|b\|\neq 0$ and $\|d\|\neq 0$, then 
\begin{equation*}
  R_{X\b{X}W\b{W}}=\|b\|^2\|d\|^2R_{X'\o{X'}W'\o{W'}},
\end{equation*}
where $X'=\tfrac{1}{\|b\|}X$ and $W'=\tfrac{1}{\|d\|}W$. If we assume that 
$X'=a'^\alpha\tfrac{\delta}{\delta z^\alpha}+b'^i\tfrac{\p}{\p v^i}$, $ W'=c'^\alpha\tfrac{\delta}{\delta z^\alpha}+d'^i\tfrac{\p}{\p v^i}$, then $b'=\tfrac{b}{\|b\|}$, and so $\|b'\|=1$. Similarly, $\|d'\|=1$. By the above two cases, we obtain $R_{X\b{X}W\b{W}}<0$ for any $k\geq k_1$. 

In summary, we can find $k_0, k_1>0$, for any $k\geq \max\{k_0,k_1\}$, and for any two non-zero $(1,0)$-type vectors $X,W$, 
we have $ R_{X\b{X}W\b{W}}<0$. This completes the proof of Theorem \ref{thm2}. 
\end{proof}

\section{Applications}\label{sec:app}

This section will consider the relative K\"ahler fibration comes from a holomorphic family of compact, canonically polarized manifolds.
We assume that $p:\mc{X}\to \mc{B}$ is a holomorphic family of compact, canonically polarized manifolds, and each fiber is equipped with a K\"ahler-Einstein metric of negative scalar curvature. One can refer to \cite{MR2969273} for more details. 

 Let 
\[
\omega_{\mathcal{X}/\mathcal{B}} = \sqrt{-1} g_{i \bar{j}}(z, v) dv^i \wedge d\bar{v}^j
\]
denote the smooth family of K\"ahler-Einstein metrics satisfying the equation
\begin{equation}\label{equ3}
\tfrac{\partial^2}{\partial v^i \partial \bar{v}^j} \log \det (g_{i \bar{j}}) = g_{i \bar{j}}.
\end{equation}
Next, we define 
\[
\omega_{\mathcal{X}} = \sqrt{-1} \partial \bar{\partial} \log \det (g_{i \bar{j}}) = \sqrt{-1} \partial \bar{\partial} \phi,
\]
where \( \phi := \log \det (g_{i \bar{j}}) \). Hence $\omega_\mc{X}$ is a relative K\"ahler form on $\mc{X}$.
 By equation \eqref{equ3}, we have
\[
e^\phi = \det (g_{i \bar{j}}) = \det (\phi_{i \bar{j}}).
\]

The geodesic curvature form satisfies the equation:
\[
(1 + \Box) c(\phi)_{\alpha \bar{\beta}} = \langle \mu_\alpha, \mu_\beta \rangle,
\]
where \[ \mu_\alpha = \bar{\partial}^V (\tfrac{\delta}{\delta z^\alpha}) = - \partial_{\bar{l}} (\phi_{\alpha \bar{j}} \phi^{\bar{j} k}) \tfrac{\partial}{\partial v^k} \otimes d\bar{v}^l \in \mathbb{H}^{0,1}(X_z, T_{X_z}) \] is harmonic, and \( \Box := - \phi^{i \bar{j}} \partial_i \partial_{\bar{j}} \); see \cite[Proposition 3]{MR2969273}. From \cite[Theorem 1]{MR2969273}, \( c(\phi) \) is semi-positive and strictly positive in the horizontal directions for families that are not infinitesimally trivial.
 If the family is effectively parametrized, then \( \Omega \) defined by 
 \begin{equation*}
  \Omega = k(p^* \omega_{\mathcal{B}}) + \omega_{\mathcal{X}},
\end{equation*}
  remains a K\"ahler metric on \( \mathcal{X} \) even for \( k = 0 \).

Note that the K\"ahler-Einstein metric on each fiber is given by \( \omega_{\mathcal{X}/\mathcal{B}} = \sqrt{-1} \phi_{i \bar{j}} dv^i \wedge d\bar{v}^j \), which induces a Hermitian metric on the relative tangent bundle \( \mathcal{V} = T_{\mathcal{X}/\mathcal{B}} \) as
\[
\left\langle \tfrac{\partial}{\partial v^i}, \tfrac{\partial}{\partial v^j} \right\rangle := \phi_{i \bar{j}}.
\]
The Hermitian vector bundle \( (T_{\mathcal{X}/\mathcal{B}}, \omega_{\mathcal{X}/\mathcal{B}}) \) is Griffiths negative if
\begin{equation}\label{equ55}
R^{\mathcal{V}}_{X \bar{X} V \bar{V}} := \left\langle R^{\mathcal{V}}(X, \bar{X}) V, V \right\rangle < 0
\end{equation}
for any non-zero vector \( X \in T^{1,0}_{(z,v)} \mathcal{X} \) and non-zero vector \( V \in \mathcal{V}_{(z,v)} \).

In particular, the geodesic curvature form satisfies
\begin{align*}
c(\phi) &= \partial \bar{\partial} \phi \left( \tfrac{\delta}{\delta z^\alpha}, \tfrac{\delta}{\delta z^\beta} \right) \sqrt{-1} dz^\alpha \wedge d\bar{z}^\beta \\
&= \partial \bar{\partial} \log \det (\phi_{i \bar{j}}) \left( \tfrac{\delta}{\delta z^\alpha}, \tfrac{\delta}{\delta z^\beta} \right) \sqrt{-1} dz^\alpha \wedge d\bar{z}^\beta \\
&= - R^{\mathcal{V}}_{\alpha \bar{\beta} i \bar{j}} \phi^{i \bar{j}} \sqrt{-1} dz^\alpha \wedge d\bar{z}^\beta,
\end{align*}
which is strictly positive in the horizontal directions by \eqref{equ55}. From Theorem \ref{thm2}, we obtain the following theorem.

\begin{theorem}\label{thm3}
Let $p:\mc{X}\to \mc{B}$ be a compact holomorphic fibration over a compact K\"ahler manifold $\mc{B}$ with negative holomorphic bisectional curvature.
Assume that the fibers are canonically polarized, and let $\omega_{\mc{X}/\mc{B}}$ denote the relative K\"ahler--Einstein metric along the fibers.
If the induced Hermitian vector bundle $(T_{\mc{X}/\mc{B}},\omega_{\mc{X}/\mc{B}})$ is Griffiths negative, then $\mc{X}$ admits a K\"ahler metric with negative holomorphic bisectional curvature.
\end{theorem}

As an application, we can solve Problem \ref{pro1} for the case where the fibers have dimension one.

\begin{corollary}
	Let \( p: \mathcal{X} \to \mathcal{B} \) be a compact holomorphic fibration over a K\"ahler manifold \( \mathcal{B} \) with negative holomorphic bisectional curvature. Suppose the fibration is a holomorphic family of compact Riemann surfaces of genus \( \geq 2 \) and is effectively parametrized. Then there exists a K\"ahler metric on \( \mathcal{X} \) with negative holomorphic bisectional curvature.
\end{corollary}

\begin{proof}
If each fiber has dimension one, then \( K_{\mathcal{X}/\mathcal{B}}^{-1} = T_{\mathcal{X}/\mathcal{B}} \), where \( K_{\mathcal{X}/\mathcal{B}}^{-1} \) denotes the anti-canonical line bundle, which is the dual of the relative canonical bundle \( K_{\mathcal{X}/\mathcal{B}} \). By \cite[Theorem 1]{MR2969273}, the Ricci curvature of \( (T_{\mathcal{X}/\mathcal{B}}, \omega_{\mathcal{X}/\mathcal{B}}) \) is negative. Therefore, \( T_{\mathcal{X}/\mathcal{B}} \) is Griffiths negative since each fiber has dimension one. The proof is complete.
\end{proof}

\bibliographystyle{alpha}
\bibliography{negativity}

\end{document}